\documentclass[preprint,12pt]{elsarticle}

\usepackage{amssymb,latexsym,amsmath}     
\usepackage{epsf,capt-of}
\usepackage{array,caption}
\usepackage{amssymb,amscd}
\usepackage{calc}
\usepackage[mathscr]{eucal}
\usepackage{hhline}
\usepackage{tikz}
\usepackage{multirow}

\usepackage{mathrsfs}
\usetikzlibrary{arrows}
\usepackage{amssymb,amsmath,amstext,amsgen,amsfonts,amsthm,mathrsfs,verbatim,url,hyperref,mathdots}
\usepackage{enumitem}

\newlist{inparaenum}{enumerate}{2}
\setlist[inparaenum]{nosep}
\setlist[inparaenum,1]{label=\bfseries\alph*.}

\newtheorem{theorem}{Theorem}[section]

\newtheorem{lemma}[theorem]{Lemma}
\newtheorem{proposition}[theorem]{Proposition}

\def\comment#1{}
 \DeclareMathOperator\diag{diag}

\def\invddots{\mathinner{\mskip1mu\raise1pt\vbox{\kern7pt\hbox{.}}\mskip2mu
		\raise4pt\hbox{.}\mskip2mu\raise7pt\hbox{.}\mskip1mu}}

\journal{TBD}

\begin{document}

\begin{frontmatter}



\title{Existence of flipped orthogonal conjugate symmetric Jordan canonical bases for real $H$-selfadjoint  matrices}


\author[label2]{S. Dogruer Akgul}
\author[label2]{A. Minenkova}
\author[label2]{V. Olshevsky}

	
\address[label2]{
	{Department of Mathematics,  University of Connecticut, Storrs CT 06269-3009,  USA. Email: sahinde.dogruer@uconn.edu, anastasiia.minenkova@uconn.edu, olshevsky@uconn.edu}}

\begin{abstract}
For real matrices selfadjoint in an indefinite inner product there are two special canonical Jordan forms, that is (i) flipped orthogonal (FO) and (ii) $\gamma$-conjugate symmetric (CS). These are the classical Jordan forms with certain additional properties induced by the fact that they are $H$-selfadjoint. In this paper we prove that for any real $H$-selfadjoint matrix there is a  $\gamma$-FOCS Jordan form that is simultaneously flipped orthogonal and  $\gamma$-conjugate symmetric.
\end{abstract}


\end{frontmatter}


\section{Introduction}

\subsection{$H$-selfadjoint matrices and the affiliation relation}

Let  $[\cdot,\cdot]$ denote the indefinite inner product, that is $[\cdot,\cdot]$ satisfies all the axioms of the usual inner product except for positivity.

For example, for every $n\times n$ invertible hermitian matrix $H$
\begin{equation}\label{Hinnerproduct}
    [x,y]_H=y^*Hx,\,\text{ for }x,y\in\mathbb{C}^n
\end{equation}
determines an indefinite inner product. Conversely, for any indefinite inner product we can find such $H$ that \eqref{Hinnerproduct} holds true. If $H$ is positive-definite then $[\cdot,\cdot]_H$ is a classical inner product, but in what follows we assume that $H$ is just Hermitian.

{\bf H-selfadjoint matrices.} A matrix $A$ is called \textbf{H-selfadjoint} if 
 \begin{equation}\label{HSA}
     A=H^{-1}A^*H\text{ (or $[Ax,y]_H=[x,Ay]_H$)}.
 \end{equation} 
These matrices have many applications and have been studied by many authors (e.g. see \cite{GLR83} and many references therein).  
 
If follows from \eqref{HSA} that $H$-selfadjoint matrices have eigenvalues symmetric about the real axis. Moreover, the sizes of Jordan blocks for the conjugate eigenvalues are the same.

{\bf Affiliation relation.  } Let us consider the change of basis    $$x\mapsto u =T^{-1}x, \quad     y\mapsto v =T^{-1}y.$$
It follows that $$[x,y]_H=y^*Hx=(Tv)^*H(Tu)=v^*\overset{G}{\overbrace{(T^*HT)}}u=[u,v]_G.$$
Therefore, if $T$ is a change of basis matrix, then in the new basis the same inner product is given by
a new congruent matrix $G = T^*HT$.

	Two pairs $(A,H)$ and $(B,G)$ are called {\bf affiliated} if $$T^{-1}AT=B\text{ and }T^{*}HT=G.$$
	The  relation $(A, H) \overset{T}{\mapsto} (B,G)$ is called the {\bf affiliation relation}.
	That is, if we change the basis, the first matrix in a new basis is similar (as in the general case), but the second is congruent as described above. 
	
	Note that that the affiliation relation preserves selfadjointness in the indefinite inner product, that is if $A$ is $H$-selfadjoint, then $B$ is $G$-selfadjoint. 
	
As to canonical forms, in	the $H$-selfadjoint case we are looking for the affiliation
    $(A,H) \mapsto (J,P)$ where not only $J$ is the Jordan form but also $P$ has a certain simple form. 
   One example of such a form is considered next. 
	
\subsection{Flipped-Orthogonal Bases}

Let us consider the following example
\begin{equation}\label{A}
    A=\begin{bmatrix}
0&1&0&0\\
0&0&1&0\\
0&0&0&1\\
-1&0&-2&0\end{bmatrix}=TJT^{-1},\quad    H=\begin{bmatrix}
0&2&0&1\\
2&0&1&0\\
0&1&0&0\\
1&0&0&0\end{bmatrix},
\end{equation}
where  $A=H^{-1}A^*H$ and
\begin{equation}\label{JT}
    J=\left[\begin{array}{cc|cc}
i&1&0&0\\
0&i&0&0\\\hline
0&0&-i&1\\
0&0&0&-i\end{array}\right],\quad     P=\left[\begin{array}{cc|cc}
0&0&0&1\\
0&0&1&0\\\hline
0&1&0&0\\
1&0&0&0\end{array}\right],
\end{equation}
{and}
\begin{equation}\label{P}
T=\frac{1}{2}\begin{bmatrix}
-i&0&i&1\\
1&-i&1&0\\
i&2&-i&1\\
-1&3i&-1&-2i\end{bmatrix}.
\end{equation}

Moreover,   $J=P^{-1}J^*P$ and $P=T^*HT$, so that
$$(A,H) \overset{T}{\mapsto} (J,P).$$
This is a very special affiliation relation, since $(J,P)$
is of a particular form with $J$ being the Jordan canonical form of $A$, and $P$ was called the sip matrix in \cite[Section 5.5]{GLR83}, which is uniquely defined up-to perturbation of the Jordan blocks in $J$.

Thus, $(J,P)$ is a canonical form. It first appeared in the works by Weierstrass \cite{W68,W95} (see also \cite[Chapter 5]{GLR83} and \cite[Chapter 3]{M63}).
It can be seen that the columns of $T=[t_i]$ satisfy 
$$[t_i,t_j]_H=[e_i,e_j]_P=\delta_{i,4-j},$$
where $\delta_{i,j}$ is the Kronecker symbol. Therefore, we suggest to call the Jordan basis $\{t_1,\ldots,t_n\}$ \textbf{flipped orthogonal} in agreement with its usage in \cite{BOP}.


\subsection{$\gamma$-Conjugate Symmetric Bases}

 While the above definitions were given for the general complex case, in this paper we focus on the case, where $A$ and $H$ are real.
 
 {\it Are there any canonical forms for the real case?}
 
 One immediate example is given next. Here the matrices $A$ and $H$ are the same as in \eqref{A}.

Let us consider another Jordan basis for $A$.
\begin{equation}\label{R}
R=\frac{1}{2}\begin{bmatrix}
-i&2&i&2\\
1&i&1&-i\\
i&0&-i&0\\
-1&i&-1&-i\end{bmatrix}.
\end{equation}
Note that the columns $R=[r_1\ r_2\ r_3\ r_4]$ are conjugate of each other:
$$r_1=\overline{r}_3\text{ and }r_2=\overline{r}_4.$$

Hence, we say that  $R$ captures a conjugate symmetric basis of $A$ from \eqref{A} and $(A,H)\overset{R}{\mapsto}(J,G)$, where
$$
G=R^*HR=\begin{bmatrix}
0&0&0&1\\
0&0&1&-3i\\
0&1&0&0\\
1&3i&0&0\end{bmatrix}
$$is not antidiagonal as was $P$ in \eqref{P}. That is, $R$ is CS but not FO.
Also, if we look at the FO basis in $T$ from \eqref{JT}, it is clearly not CS.

\subsection{FOCS Bases}
Now that we have introduced the FO and $\gamma$-CS canonical bases we might ask:
{\it Is there a Jordan basis for real $H$-selfadjoint matrices that is simultaneously  flipped orthogonal and $\gamma$-conjugate symmetric?} 

Consider the following matrix whose columns also capture a Jordan basis of the same matrix $A$ as above.
$$M=\frac{1}{4}\begin{bmatrix}
-2i&1&2i&1\\
2&-i&2&i\\
2i&3&-2i&3\\
-2&5i&-2&-5i\end{bmatrix}.$$
In fact, $(A,H)\overset{M}{\mapsto}(J,P)$. Moreover, $M$ is CS.

We summarize all the above examples in the following figure.

\begin{center}
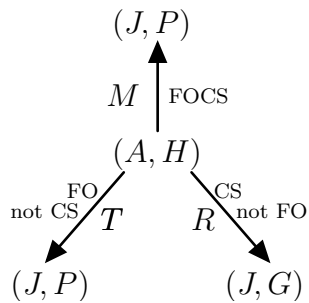

\begin{tikzpicture}[line cap=round,line join=round,>=triangle 45,x=1.5cm,y=1.1cm]
\clip(-14.5,4.35) rectangle (-11.5,7.9);
\draw (-13.5,8) node[anchor=north west] {$(J,P)$};
\draw [->,line width=1pt] (-13,6.4) -- (-13,7.5);
\draw [->,line width=1pt] (-13.3,5.9) -- (-14,4.8);
\draw [->,line width=1pt] (-12.7,5.9) -- (-12,4.8);
\draw (-12.5,4.85) node[anchor=north west] {$(J,G)$};
\draw (-14.4,4.85) node[anchor=north west] {$(J,P)$};
\draw (-13.5,6.45) node[anchor=north west] {$ (A,H)$};
\draw (-13.55,7.1) node[anchor=north west] {$M$};
\draw (-13,7.05) node[anchor=north west] {\scriptsize FOCS};
\draw (-13.6,5.6) node[anchor=north west] {$T$};
\draw (-13.9,5.9) node[anchor=north west] {\scriptsize FO};
\draw (-14.4,5.65) node[anchor=north west] {\scriptsize not CS};
\draw (-12.4,5.65) node[anchor=north west] {\scriptsize not FO};
\draw (-12.6,5.9) node[anchor=north west] {\scriptsize CS};
\draw (-13.6,5.6) node[anchor=north west] {$T$};
\draw (-12.8,5.6) node[anchor=north west] {$R$};
\end{tikzpicture}
\captionof{figure}{Jordan bases of $A$.}
\end{center}

In this case, we say that $M$ captures the \textbf{flipped orthogonal conjugate symmetric basis}.

This is only an example. We show the existence of such bases in the next section.

\subsection{Main results}
In this paper, we consider the case of real matrices $A$ and $H$. For the real case we prove the existence of a \lq\lq refined\rq\rq~FO-basis which we  call the $i$-FOCS basis (i.e. flipped orthogonal and also $i$-conjugate symmetric).  Namely, the Jordan chains of $A$ corresponding to $\lambda$ and $\overline{\lambda}$ are just scaled conjugates of each other, still enjoying the flipped orthogonal property. 


\section{Three Canonical Bases}

We have already presented an example of an FO basis. The next subsection shows the existence of such a basis for any $H$-selfadjoint matrix.

\subsection{FO Bases}

The next proposition establishes the existence of FO bases for any (potentially non-real) pair $(A,H)$.

\begin{proposition}[\cite{BOP}]\label{thm:main1} Let ${A},H \in\mathbb{C}^{n \times n}$  be given with $H=H^*$ and $A$ being $H$-selfadjoint matrix. Let $J$ be its Jordan form
	\begin{equation}\label{J}
	J= J(\lambda_1)\oplus \cdots J(\lambda_\alpha)\oplus \widehat{J}(\lambda_{\alpha+1})\oplus \cdots \oplus\widehat{ J}(\lambda_\beta)
	\end{equation}
where $\lambda_1,\ldots,\lambda_\alpha\in\mathbb{R}$,  $\lambda_{\alpha+1},\ldots,\lambda_\beta\notin\mathbb{R}$,  $J(\lambda_i)$ is a  Jordan  block for real eigenvalues $\lambda_1,\ldots,\lambda_{\alpha}$ and 
	\begin{equation}\label{Jk}
	\widehat{J}(\lambda_k)=\left [ \begin{array}{cc}
	J(\lambda_k) &0\\
	0 &J(\overline{\lambda}_k)\\
	\end{array} \right]
	\end{equation} is the direct sum of two Jordan blocks of the same size corresponding to $\lambda_k$ and~$\overline{\lambda}_k$.  Then there exists an invertible matrix $T$ such that
\begin{equation*}
    (A,H)\overset{T}{\mapsto} (J,P)
\end{equation*} where
	\begin{equation}\label{eq:P}
	P=P_1 \oplus\cdots \oplus P_{\alpha} \oplus {P}_{\alpha+1} \oplus \cdots \oplus{P}_{\beta}
	\end{equation}
	where $P_k$ is a sip matrix $\epsilon_k \tilde{I}_k$
	(i.e. $\left[ \begin{smallmatrix}
	&&& \\
	0 & \ldots & 0 &\epsilon_k\\
	\vdots & \invddots & \epsilon_k & 0\\
	0 & \invddots & \invddots & \vdots\\[0.5em]
	\epsilon_{k} & 0 & \cdots & 0\\
	&&& 
	\end{smallmatrix} \right]$) of the same size as $J({\lambda_k})$ and $\epsilon_{k}=\pm{1}$  for $ k=1,\ldots,\alpha$ and a sip matrix $\tilde{I}_k$ (i.e. $\left[ \begin{smallmatrix}
	&&& \\
	0 & \ldots & 0 &1\\
	\vdots & \invddots & 1 & 0\\
	0 & \invddots & \invddots & \vdots\\[0.5em]
	1 & 0 & \cdots & 0\\
	&&& 
	\end{smallmatrix} \right]$) of the same size as  $\hat{J}({\lambda_k})$ for $k=\alpha+1,\ldots,\beta$. 
\end{proposition}
\textit{Why flipped orthogonal?}

If $t_i,t_j$ are the $i$th, $j$th columns of $T$ respectively, then
$$[t_i,t_j]_H=[e_i,e_j]_P.
$$
In other words, if we partition $T$ in the following way
	\begin{displaymath}
	T= [T_1  \ldots T_{\alpha}\ | T_{\alpha+1} \ldots T_{\beta} ],
	\end{displaymath}
 for  $k=1,\ldots, \alpha$ the columns of 
		\begin{displaymath}
	T_k=[g_{0,k}, \ldots, g_{{p_k}-1,k}]
	\end{displaymath}
	form the Jordan chain of $A$ corresponding to real $\lambda_k$. Similarly, the columns of
		\begin{displaymath}
	T_k=[\underbrace{g_{0,k}, \ldots, g_{{p_k}-1,k}}_{\lambda_k}\ | \underbrace{h_{0,k},\ldots,h_{{p_k}-1,k}}_{\overline{\lambda}_k} ]
	\end{displaymath}
	for $k=\alpha+1,\ldots, \beta$ form two Jordan chains, corresponding to $\lambda_k$ and ${\overline{\lambda}_k}$ respectively. Then the structure of the signature  matrix $P$ implies the following orthogonality relation,
	\begin{displaymath}
	g_{ik}^*Hg_{jm} =0, \;\; \;\;
	h_{ik}^*Hg_{jm} =0, \;\; \;\;
	h_{ik}^*Hh_{jm} =0, \;\; \text{ for }\;\; k \neq m.
	\end{displaymath}
	Further, we have that
	\begin{displaymath}
	g_{i,k}^*Hg_{j,k} = \left \{  \begin{array}{ll}
	\epsilon_k, \quad j=p_{k} -1-i\\ 
	0 \quad \text{otherwise}
	\end{array} \right.  \qquad \text{for}  \quad k=1,\ldots,\alpha
	\end{displaymath}
	\begin{displaymath}
	\left \{  \begin{array}{ll}
	g_{ik}^*Hg_{jk} =0,\\
	h_{ik}^*Hh_{jk} =0,\\ 
	h_{i,k}^*Hg_{j,k} = \left \{  \begin{array}{ll}
	1, \quad j=p_{k} -1-i\\ 
	0 \quad \text{otherwise}
	\end{array} \right. 
	\end{array} \right.  \qquad \text{for}  \quad k=\alpha+1,\ldots,\beta.
	\end{displaymath}
 In terms of our example it means that
 $$
\begin{bmatrix}
[t_1,t_1]_H&[t_1,t_2]_H&[t_1,t_3]_H&[t_1,t_4]_H\\
[t_2,t_1]_H&[t_2,t_2]_H&[t_2,t_3]_H&[t_2,t_4]_H\\
[t_3,t_1]_H&[t_3,t_2]_H&[t_3,t_3]_H&[t_3,t_4]_H\\
[t_4,t_1]_H&[t_4,t_2]_H&[t_4,t_3]_H&[t_4,t_4]_H\end{bmatrix}
    =$$
    $$\begin{bmatrix}
[e_1,e_1]_P&[e_1,e_2]_P&[e_1,e_3]_P&[e_1,e_4]_P\\
[e_2,e_1]_P&[e_2,e_2]_P&[e_2,e_3]_P&[e_2,e_4]_P\\
[e_3,e_1]_P&[e_3,e_2]_P&[e_3,e_3]_P&[e_3,e_4]_P\\
[e_4,e_1]_P&[e_4,e_2]_P&[e_4,e_3]_P&[e_4,e_4]_P\end{bmatrix}
    =\begin{bmatrix}
0&0&0&1\\
0&0&1&0\\
0&1&0&0\\
1&0&0&0\end{bmatrix}=P.
 $$

As we showed by our example in Introduction -- not every Jordan basis is necessarily flipped-orthogonal. However, the above proposition yields that the flipped orthogonal basis always exists for $H$-selfadjoint matrices. In fact, in view of the structure of $P$, it is captured by the columns of $T$.

To show that $\gamma$-FOCS bases exist, we need to introduce yet another canonical form.

\subsection{$\gamma$-CS Bases}

Revising example in \eqref{eq:R}, one question arises.

\textit{Why do the columns of $R$ from our example in Introduction capture the Jordan basis?}

For any real matrix $A$, if $\lambda$ is its non-real eigenvalue, then so is $\overline{\lambda}$. Moreover, if $g_k\to g_{k-1}\to \ldots\to g_0\to0$ is a Jordan chain corresponding to $\lambda$ then 
$$
\gamma \overline{q_{k-1}}=\gamma\overline{(A-\lambda I)q_k}={(\overline{A}-{\overline{\lambda}} \overline{I}){\gamma\overline{q_k}}}={({A}-{\overline{\lambda}} {I}){\gamma\overline{q_k}}},
$$
i.e. $\gamma\overline{g}_k\to\gamma \overline{g}_{k-1}\to \ldots\to \gamma\overline{g}_0\to0$  is also a Jordan chain of $A$ but corresponding to $\overline{\lambda}$ for any non-zero $\gamma \in \mathbb{C}$.

This observation leads to the following definition.

	Suppose $A\in \mathbb{R}^{ n \times n} $ and that there exists an invertible matrix $N$ such that $A=NJN^{-1}$, where 
	\begin{equation*}
	J=J(\lambda_{1})\oplus\cdots \oplus J(\lambda_{\alpha})\oplus \widehat{J}(\lambda_{\alpha+1})\oplus \cdots \oplus \widehat{J}(\lambda_{\beta})
	\end{equation*}is the Jordan canonical form and $\widehat{J}(\lambda_k)$ defined in \eqref{Jk}.
	The matrix
	\begin{equation}{\label{N}}
	N=[N_1 \lvert \ldots \lvert N_{\alpha} \lvert {N}_{\alpha+1} \lvert \ldots \lvert {N}_{\beta}]
	\end{equation}
	can be chosen that
	\begin{equation}\label{eq:t1}
	N_{k}=[Q_{k} \lvert \gamma\overline{Q}_{k}] 
	\end{equation} 
	for any $\gamma\neq0$ and $k=\alpha+1,\ldots,\beta$. The columns of $N$ capture the Jordan basis of $A$, and because of \eqref{eq:t1} we call it the \textbf{$\gamma$-conjugate symmetric basis}.

In this paper we show that for any real $H$-selfadjoint matrix $A$ there exist a Jordan basis that is simultaneously FO and $i$-CS.  For this we need to consider the real canonical form.

\subsection{{Real Canonical Form} }

The following well-known  result (see~\cite[Theorem 6.1.5]{GLR83}) describes a purely real relation

\begin{equation}\label{eq:mapsR}
(A,H)\overset{R}{\mapsto}(J_R,P)
\end{equation}
where all five matrices are real.

\begin{proposition}\label{thm:real main}
 Let ${A},H \in\mathbb{R}^{n \times n}$  be given with $H=H^\top$ and $A$ being $H$-selfadjoint matrix. Let $J_R$ be its real Jordan form
	\begin{equation}\label{eq:real jordan}
	J_R= J(\lambda_1)\oplus \cdots \oplus J(\lambda_\alpha)\oplus \widehat{J}_R(\lambda_{\alpha+1})\oplus \cdots \oplus\widehat{J}_R(\lambda_\beta),
	\end{equation}
where $\lambda_1,\ldots,\lambda_\alpha\in\mathbb{R}$,  $\lambda_{\alpha+1},\ldots,\lambda_\beta\notin\mathbb{R}$,  $J(\lambda_i)$ is a  Jordan  block for real eigenvalues $\lambda_1,\ldots,\lambda_{\alpha}$ and 
		\begin{equation*}
	\widehat{J}_R(\lambda_k)=
	{
		\left [ \begin{smallmatrix}
		\sigma_k & \tau_k   & 1        &      0 & 0    & \cdots&      & \cdots&  & &0 \\
		-\tau_k  & \sigma_k & 0        &      1 & 0    &     \ddots  &      &       &  & &  \\ 
		0& 0        & \sigma_k & \tau_k & 1    &\ddots &   \ddots   &       &  & &\vdots\\ 
		0 & 0   &-\tau_k    &\sigma_k  &0&1 &\ddots&\ddots& &  \\
		&  \ddots        &  \ddots  & \ddots &\ddots  &\ddots&\ddots &\ddots &\ddots & & \vdots\\
		\vdots  &  &\ddots  &\ddots & \ddots&\ddots   & \ddots&\ddots &\ddots & \ddots&\\
		&  &  & \ddots&\ddots & \ddots    & \ddots& \ddots& \ddots & \ddots&  0\\
		&  &  &  &\ddots& \ddots&   0&\sigma_k & \tau_k&1 &0\\
		\vdots &  &  & & &   \ddots &\ddots &-\tau_k & \sigma_k & 0&1\\
		&  &  & & & &    \ddots &0 & 0 &\sigma_k &\tau_k\\[0.6em]
		0 &  & \ldots & &  &\ldots &  & 0& 0 &-\tau_k &\sigma_k\\
		\end{smallmatrix} \right]},
	\end{equation*} where $\lambda_k=\sigma_k+i\tau_k$. Then there exists an invertible real matrix $R$ such that $(A,H)\overset{R}{\mapsto}(J_R,P)$ where
	$P$ is a sip matrix of the form 
	\begin{displaymath}
	P=P_1 \oplus\cdots \oplus P_{\alpha} \oplus {P}_{\alpha+1} \oplus \cdots \oplus{P}_{\beta},
	\end{displaymath}
	where $P_k$ is a sip matrix $\epsilon_k \widetilde{I}$ of the same size as $J({\lambda_k})$ for $ k=1,\ldots,\alpha$ and a sip matrix $\widetilde{I}$ of the same size as  $\hat{J}_R({\lambda_k})$ for $k=\alpha+1,\ldots,\beta$ and $\epsilon_{k}=\pm{1}$ for $k=1,\ldots,\alpha$. 
\end{proposition}
Suppose that $A \in \mathbb{R}^{n \times n}$ and $J_R$ is the real Jordan form in \eqref{eq:real jordan}. We say that columns of matrix $R$ in \eqref{eq:mapsR} form a {\bf real canonical basis} of $A$.

Finally, we are ready to prove the main result.


\section{Existence of $i$-FOCS Bases}
The FO basis was defined for pairs of matrices $(A,H),$ where $A$ and  $H$ are not necessarily real. The $\gamma$-CS basis is defined  for the case of   $A$ being a real matrix.
\textit{So what about real $H$-selfadjoint matrices?}

Proposition {\ref{thm:main1}} implies the existence of the FO basis for $(A,H)$, and the latter, generally, does not have the CS property.

\begin{theorem}{\bf (Existence of  an $i$-FOCS basis)} \label{thm:focs1} Let $A$ be a real $H$-selfadjoint matrix where $H$ is real, invertible, and Hermitian. Then $N$ in $(A,H)\overset{N}{\mapsto}(J,P)$ can be chosen $i$-conjugate symmetric and flipped orthogonal at the same time.
\end{theorem}

To establish this result, let us discuss the relation between RC  and $i$-FOCS bases. 
As we will see, for an arbitrary real pair $(A,H)$, the relation between $i$-FOCS and RC bases is carried over with the help of the same fixed matrix $S$.
\begin{lemma}\label{thm:focs-rc}  Let $A,H \in \mathbb{R}^{n \times n}$, $H=H^\top$, and $A$ is $H$-selfadjoint, and let the columns of  $R$  capture the RC basis of $A$. 
	\begin{equation}\label{eq:r}
	R=[N_1, \ldots, N_{\alpha} \lvert K_{\alpha+1}, \ldots, K_{\beta}]
	\end{equation}
	capture the RC basis of $A$. 
	Then for  $S=\diag({I_1,\ldots,I_{\alpha} | S_{\alpha+1},\ldots,S_{\beta}})$, where the explicit formula for $S_j$ is:
	\begin{displaymath}
	S_j=\frac{1}{\sqrt{2}}
	\left [ \begin{array}{ccccccccccccc}
	1&0&\cdots &0&i&0 &\cdots&0\\
	i&0&{}&\vdots&1&0&{}&\vdots\\\hline
	0&1&\ddots&0 &0&i&\ddots&0\\
	0&i&\ddots&\vdots&  0 &1&\ddots&\vdots \\\hline
	\vdots&\ddots &\ddots&0&\vdots&\ddots &\ddots&0\\\hline
	\vdots&\ddots &\ddots&1&\vdots&\ddots &\ddots&i\\
	0&\cdots & 0&i&0&\cdots&0 & 1\\
	\end{array} \right],
	\end{displaymath} 
	the columns of $N=RS$ capture the {\it{i}}-FOCS basis of $A$.
	\end{lemma}
	
The next result is the converse to the one in Lemma \ref{thm:focs-rc}
\begin{lemma} If the columns of
	\begin{equation}\label{eq:t}
	N=[N_1, \ldots, N_{\alpha} \lvert N_{\alpha+1}, \ldots, N_{\beta}] \text{  and }
	N_k=[Q_k \lvert i\overline{Q}_k], \quad k=\alpha+1, \ldots, \beta
	\end{equation}
	capture the {\it{i}}-FOCS basis of $(A,H)$. 
	
	Further, note that $S^{-1}=\diag({I_1,\ldots,I_{\alpha} | S^{-1}_{\alpha+1},\ldots,S^{-1}_{\beta}})$, where 
	
	\begin{displaymath}
	S_j^{-1}=\frac{1}{\sqrt{2}}
	\left [ \begin{array}{ccccccccc}
	1&-i&0&0&\cdots& \cdots &0\\
	0&0&1&-i&0&\cdots & 0\\
	\vdots & \ddots   & \ddots& \ddots & \ddots & \ddots &\vdots\\
	\vdots & {}   & \ddots& \ddots & -i& 0 &0\\
	0&\cdots&\cdots& 0&0&1&-i\\
	-i&1&0&0&\cdots& \cdots &0\\
	0&0&-i&1&0&\cdots & 0\\
	\vdots & \ddots   & \ddots& \ddots & \ddots & \ddots &\vdots\\
	\vdots & {} & \ddots& \ddots & 1& 0 &0\\
	0& \cdots&\cdots& 0&0&-i&1\\
	\end{array} \right]\\.
	\end{displaymath}Then the matrix 
	
	\begin{equation}\label{eq:R}
	R=NS^{-1}
	\end{equation}is real and its columns  
	capture the RC basis of $A$.  
	Moreover,  from \eqref{eq:r} we get $K_{j}=N_{j}S_j^{-1}$ for $j=\alpha+1, \ldots, \beta$. 
	
\end{lemma}

\begin{proof}[Proof of Lemma \ref{thm:focs-rc}]
	First, let us show that the columns of  matrix $N$ in \eqref{eq:R} capture a basis  of $A$, then that this basis is $i$-CS . 	 Without loss of generality we  consider the case of a single real Jordan block, assuming that $\sigma(A)=\big\{\lambda, \overline{\lambda}\big\}$.
		The relation $S^{-1}{J_R}S={J}$  follows from		
	\begin{displaymath}{J}S=\frac{1}{\sqrt{2}}
	{
		\left [ \begin{smallmatrix}
		\sigma+i\tau&1 &0 &\dots&0&i\sigma+\tau&i &0 &\dots&0\\
		-\tau+i\sigma&i&0&\dots&0 &-i\tau+\sigma&1&0&\dots&0 \\
		0&\sigma+i\tau&1 &\dots&0&0&i\sigma+\tau&i &\dots&0\\
		0&-\tau+i\sigma&i&\dots&0 &0&-i\tau+\sigma&1&\dots&0 \\
			\vdots&\vdots&\vdots&\vdots &\vdots&\vdots&\vdots&\vdots&\vdots&\vdots\\
		0&0 &0 &\dots&1&0&0 &0 &\dots&i\\
		0&0&0&\dots&i &0&0&0&\dots&1 \\
		0&0&0 &\dots&\sigma+i\tau&0&0&0 &\dots&i\sigma+\tau\\
		0&0&0&\dots&-\tau+i\sigma&0&0&0&\dots&-i\tau+\sigma \\
		\end{smallmatrix} \right]}=S{J}_R,
	\end{displaymath}
	where $\lambda=\sigma+i\tau$. Therefore,
	$$A=RJ_RR^{-1}=(RS)(S^{-1}J_RS)(RS)^{-1}=NJN^{-1}.
	$$
	That is the columns of $N$ indeed capture a basis of $A$. For a pair of vectors $a$ and $b$
	$$i\cdot\overline{a+ib}=i(a-ib)=ia+b.$$
	This is exactly what multiplication by $S$ on the left does to real columns of $R$. Hence, we get the $i$-conjugate symmetry for $N$.
	
	Next, we want to show that the columns of $N$ form an FO basis.	The relation $S^{\ast}PS=P$ in  \eqref{eq:mapsS} follows from
	\begin{displaymath}
	S^{*}PS=P \text{ being equivalent to } PS^{*}PS=PP=I.
	\end{displaymath}
	Thus, we just need to show that $PS^{*}PS=I$.
	\begin{displaymath}PS^{*}PS=\frac{1}{2}\left [ \begin{smallmatrix}
	&&&&&&&\\
	0&0&\dots&0&i&0&0&\dots&1\\
	0&0&\dots&0&1&0&0&\dots&i\\
	0&0&\dots&i&0&0&0&\dots&0\\
	0&0&\dots&1&0&0&0&\dots&0\\
	\vdots&\vdots&\iddots&\vdots&\vdots&\vdots&\vdots&\iddots&\vdots\\
	0&i&\dots&0&0&0&1&\dots&0\\
	0&1&\dots&0&0&0&i&\dots&0\\
	i&0&\dots&0&0&1&0&\dots&0\\
	1&0&\dots&0&0&i&0&\dots&0\\
	&&&&&&&\\
	\end{smallmatrix} \right]\\
	\left [ \begin{smallmatrix}
	0&0&0&0&\dots&-i&1\\
	\vdots&\vdots&\vdots&\vdots&\iddots&\vdots&\vdots\\
	0&0&-i&1&\dots&0&0\\
	-i&1&0&0&\dots&0&0\\
	0&0&0&0&\dots&1&-i\\
	\vdots&\vdots&\vdots&\vdots&\iddots&\vdots&\vdots\\
	0&0&1&-i&\dots&0&0\\
	1&-i&0&0&\dots&0&0\\
	\end{smallmatrix} \right]\\=I.
	\end{displaymath}
	
	Since $(A,H)\overset{R}{\mapsto}(J_R,P)$, the above relation implies that
		\begin{equation}\label{eq:mapsS}
	({J_R},P)\overset{S}{\mapsto}({J},P).
	\end{equation}
	Hence, $(A,H)\overset{N}{\mapsto}(J,P)$ and the basis in question is FO as well as $i$-CS, i.e. it is $i$-FOCS.
	
\end{proof} 


	The converse statement follows from the relation $(J,P)\overset{S^{-1}}{\mapsto}(J_R,P)$, using similar argument.

Combining Proposition \ref{thm:real main} and Lemma \ref{thm:focs-rc}, we  get the result of Theorem~\ref{thm:focs1}.

\end{document}